\documentclass{article}%
\usepackage{amsfonts}
\usepackage{amsmath}
\usepackage[ignoreall]{geometry}
\usepackage{graphicx}
\usepackage{amssymb}%
\providecommand{\U}[1]{\protect\rule{.1in}{.1in}}
\newtheorem{theorem}{Theorem}

\newtheorem{conjecture}[theorem]{Conjecture}

\newtheorem{lemma}[theorem]{Lemma}

\newenvironment{proof}[1][Proof]{\noindent\textbf{#1.} }{\ \rule{0.5em}{0.5em}}
\graphicspath{{D:/Radovi/Remoteness/Slike/}}
\begin{document}

\title{Remoteness, proximity and few other distance invariants in graphs}
\author{Jelena Sedlar\\\\University of Split, Faculty of civil engeneering, architecture and geodesy, \\Matice hrvatske 15, HR-21000, Split, Croatia.\\\\jsedlar@gradst.hr}
\date{}
\maketitle

\begin{abstract}
We establish maximal trees and graphs for the difference of average distance
and proximity proving thus the corresponding conjecture posed in
\cite{Aouchiche(2011)networks}. We also establish maximal trees for the
difference of average eccentricity and remoteness and minimal trees for the
difference of remoteness and radius proving thus that the corresponding
conjectures posed in \cite{Aouchiche(2011)networks} hold for trees.

\end{abstract}

\noindent\newline\textbf{AMS Subject Classification:} 05C35\newline%
\newline\textbf{Keywords:} Remoteness, Proximity, Distance invariants,
Extremal trees

\section{Introduction}

All graphs $G$ in this paper are simple and connected. Vertex set of graph $G
$ will be denoted with $V,$ edge set with $E.$ Number of vertices in $G$ is
denoted with $n,$ number of edges with $m.$ A path on $n$ vertices will be
denoted with $P_{n},$ while $C_{n}$ will denote a cycle on $n$ vertices. A
\emph{tree} is the graph with no cycles, and a \emph{leaf} in a tree is any
vertex of degree $1.$

The \emph{distance} $d(u,v)$ between two vertices $u$ and $v$ in $G$ is
defined as the length of a shortest path connecting vertices $u$ and $v$. The
average distance between all pairs of vertices in $G$ is denoted with $\bar
{l}$. The \emph{eccentricity} $e(v)$ of a vertex $v$ in $G$ is the largest
distance from $v$ to another vertex of $G$. The \emph{radius} $r$ of graph $G$
is defined as the minimum eccentricity in $G$, while the \emph{diameter} $D$
of $G$ is defined as the maximum eccentricity in $G.$ The average eccentricity
of $G$ is denoted with $ecc$. That is
\[
r=\min_{v\in V}e(v),\text{ }D=\max_{v\in V}e(v),\text{ }ecc=\frac{1}{n}%
\sum_{v\in V}e(v).
\]
The \emph{center} of a graph is a vertex $v$ of minimum eccentricity. It is
well-known that every tree has either only one center or two centers which are
adjacent. The \emph{diametric path} in $G$ is a shortest path from $u$ to $v,$
where $d(u,v)$ is equal to the diameter of $G$.

The \emph{transmission} of a vertex $v$ in a graph $G$ is the sum of the
distances between $v$ and all other vertices of $G$. The transmission is said
to be \emph{normalized} if it is divided by $n-1$. Normalized transmission of
a vertex $v$ will be denoted with $\pi(v).$ The \emph{remoteness} $\rho$ is
defined as the maximum normalized transmission, while \emph{proximity} $\pi$
is defined as the minimum normalized transmission. That is
\[
\pi=\min_{v\in V}\pi(v),\text{ }\rho=\max_{v\in V}\pi(v).
\]
In other words, proximity $\pi$ is the minimum average distance from a vertex
of $G$ to all others, while the remoteness $\rho$ of a graph $G$ is the
maximum average distance from a vertex of $G$ to all others. These two
invariants were introduced in \cite{Aouchiche(2006)thesis},
\cite{Aouchiche(2007)match}. A vertex $v\in V$ is \emph{centroidal} if
$\pi(v)=\pi(G)$, and the set of all centroidal vertices is the \emph{centroid}
of $G$.

Recently, these concepts and relations between them have been extensively
studied (see \cite{Aouchiche(2006)thesis}, \cite{Aouchiche(2007)match},
\cite{Aouchiche(2010)Nordhaus}, \cite{Aouchiche(2011)networks},
\cite{Baoyindureng(2012)}, \cite{Sedlar(2008)}). For example, in
\cite{Aouchiche(2010)Nordhaus} the authors established the Nordhaus--Gaddum
theorem for $\pi$ and $\rho$. In \cite{Aouchiche(2011)networks} upper and
lower bounds for $\pi$ and $\rho$ were obtained expressed in number $n$ of
vertices in $G$. Also, relations of both invariants with some other distance
invariants (like diameter, radius, average eccentricity, average distance,
etc.) were studied. The authors posed several conjectures (one of which was
solved in \cite{Baoyindureng(2012)}), among which the following.

\begin{conjecture}
\label{Con1}Among all connected graphs $G$ on $n\geq3$ vertices with average
distance $\bar{l}$ and proximity $\pi$, the difference $\bar{l}-\pi$ is
maximum for a graph $G$ composed of three paths of almost equal lengths with a
common end point.
\end{conjecture}

\begin{conjecture}
\label{Con2}Let $G$ be a connected graph on $n\geq3$ vertices with remoteness
$\rho$ and average eccentricity $ecc$. Then%
\[
ecc-\rho\leq\left\{
\begin{tabular}
[c]{ll}%
$\frac{3n+1}{4}\frac{n-1}{n}-\frac{n}{2}$ & if $n$ is odd,\\
$\frac{n-1}{4}-\frac{1}{4n-4}$ & if $n$ is even,
\end{tabular}
\right.
\]
with equality if and only if $G$ is a cycle $C_{n}$.
\end{conjecture}

\begin{conjecture}
\label{Con3}Let $G$ be a connected graph on $n\geq3$ vertices with remoteness
$\rho$ and radius $r$. Then%
\[
\rho-r\geq\left\{
\begin{tabular}
[c]{ll}%
$\frac{3-n}{4}$ & if $n$ is odd,\\
$\frac{n^{2}}{4n-4}-\frac{n}{2}$ & if $n$ is even.
\end{tabular}
\right.
\]
The inequality is best possible as shown by the cycle $C_{n}$ if $n$ is even
and by the graph composed by the cycle $C_{n}$ together with two crossed edges
on four successive vertices of the cycle.
\end{conjecture}

In this paper we prove Conjecture \ref{Con1}, and find the extremal trees for
$ecc-\rho$ and $\rho-r$ (maximal and minimal trees respectively) showing thus
that Conjectures \ref{Con2} and \ref{Con3} hold for trees.

All these conjectures were obtained with the use of AutoGraphiX, a
conjecture-making system in graph theory (see for example
\cite{Caporossi(2000)} and \cite{Caporossi(2004)}). Also, some results on
center and centroidal vertices will be used which are already known in
literature since those concepts were also quite extensively studied (see for
example \cite{Bandelt(1984)}, \cite{Chang(1991)}, \cite{Chartrand(1993)},
\cite{Jordan(1869)}).

\section{Preliminaries}

Let us introduce some additional notation for trees and state some auxiliary
results known in literature. First, we will often use a notion of diametric
path. So, if a tree $G$ of diameter $D$ has diametric path $P,$ we will
suppose that vertices on $P$ are denoted with $v_{i}$ so that $P=v_{0}%
v_{1}\ldots v_{D}.$ When deleting edges of $P$ from $G,$ we obtain several
connected components which are subtrees rooted in vertices of $P.$ Now,
$G_{i}$ will denote a connected component of tree $G\backslash P$ rooted in
vertex $v_{i}$ of $P$ and $V_{i}$ will denote set of vertices of $G_{i}.$

Furthermore, for a tree $G$ let $e\in E$ be an edge in $G$ and $u\in V$ a
vertex in $G.$ With $G_{u}(e)$ we will denote the connected component of $G-e
$ containing $u.$ Also, we denote $V_{u}(e)=V(G_{u}(e))$ and $n_{u}%
(e)=\left\vert V_{u}(e)\right\vert .$ Now the following lemma holds.

\begin{lemma}
\label{lemma0Reference}The following statements hold for a tree $G$:

\begin{enumerate}
\item a vertex $v\in V(G)$ is a centroidal vertex if and only if for any edge
$e$ incident with $v$ holds $n_{v}(e)\geq\frac{n}{2},$

\item $G$ has at most two centroidal vertices,

\item if there are two centroidal vertices in $G$, then they are adjacent,

\item $G$ has two centroidal vertices if and only if there is an edge $e$ in
$G$, such that the two components of $G-e$ have the same order. Furthermore,
the end vertices of $e$ are the two centroidal vertices of $G$.
\end{enumerate}
\end{lemma}

\begin{proof}
See \cite{Baoyindureng(2012)}.
\end{proof}

Also, we will often use transformation of tree $G$ to $G^{\prime}.$ For the
sake of notation simplicity, we will write $D^{\prime}$ for $D(G^{\prime}), $
$\rho^{\prime}$ for $\rho(G^{\prime}),$ $\pi^{\prime}(v)$ for $\pi(v)$ in
$G^{\prime},$ etc.

\section{Average distance and proximity}

To prove Conjecture \ref{Con1} for trees, we will use graph transformations
which transform tree to either:

1) path $P_{n}$,

2) a tree consisting of four paths of equal length with a common end point,

3) a tree consisting of three paths of almost equal length with a common end point.

So let us first prove that among those graphs the difference $\bar{l}-\pi$ is
maximum for the last.

\begin{lemma}
\label{lemma1Path}The difference $\bar{l}-\pi$ is greater for a tree $G$ on
$n$ vertices consisting of three paths of almost equal length with a common
end point than for path $P_{n}$.
\end{lemma}

\begin{proof}
By direct calculation.
\end{proof}

\begin{lemma}
\label{lemma1Cross}The difference $\bar{l}-\pi$ is greater for a tree $G$ on
$n$ vertices consisting of three paths of almost equal length with a common
end point than for a tree $G^{\prime}$ on $n$ vertices consisting of four
paths of equal length.
\end{lemma}

\begin{proof}
First note that number of vertices $n$ must be odd number, moreover $n=4k+1.$
For a tree $G^{\prime}$ by a simple calculation we establish%
\[
\bar{l}(G^{\prime})=\frac{3n^{2}+10n+3}{16n},\text{ }\pi(G^{\prime}%
)=\frac{n+3}{8}.
\]
Since for a path $P_{n}$ on odd number of vertices holds
\[
\bar{l}(P_{n})=\frac{n+1}{3},\text{ }\pi(P_{n})=\frac{n+1}{4},
\]
it is easily verified that for $n\geq9$ the difference $\bar{l}-\pi$ is
greater or equal for $P_{n}$ than for $G^{\prime},$ so the claim follows
from\ Lemma \ref{lemma1Path}. It only remains to prove the case $n=5$, which
is easily done by direct calculation.
\end{proof}

Now, let us introduce a transformation of a general tree which decreases
number of leafs in tree, but increases $\bar{l}-\pi.$

\begin{lemma}
\label{lemma1General}Let $G$ be a tree on $n$ vertices with at least four
leafs. Then there is a tree $G^{\prime}$ on $n$ vertices with three leafs for
which the difference $\bar{l}-\pi$ is greater or equal than for $G$.
\end{lemma}

\begin{proof}
Let $u$ be centroidal vertex of $G$, let $v$ be the branching vertex furthest
from $u.$ We distinguish two cases.

CASE I: $u\not =v.$ Let $G_{v}$ be a subtree of $G$ rooted in $v$ consisting
of all vertices $w$ such that path from $u$ to $w$ leads through $v.$ Since
$v$ is branching vertex furthest from $u,$ tree $G_{v}$ consists of paths with
a common end $v.$ Let $P_{1}$ and $P_{2}$ be two such paths. For $i=1,2$ let
$x_{i}$ be a vertex in $P_{i}$ adjacent to $v$ and let $y_{i}$ be a leaf in
$P_{i}.$ Let $G^{\prime}$ be a tree obtained from $G$ by deleting edge
$vx_{2}$ and adding edge $x_{2}y_{1}.$ This transformation is illustrated in
Figure \ref{Figure1}. Note that $G^{\prime}$ has one leaf less than $G.$ We
want to prove that the difference $\bar{l}-\pi$ is greater for $G^{\prime}$
then for $G.$ For that purpose let us denote $d_{1}=d(v,y_{1})$ and
$d_{2}=d(v,y_{2}).$ Note that
\[
\pi(G^{\prime})\leq\pi^{\prime}(u)=\pi(u)+\frac{d_{1}d_{2}}{n-1}=\pi
(G)+\frac{d_{1}d_{2}}{n-1}.
\]
Also,
\[
\bar{l}(G^{\prime})=\bar{l}(G)+\frac{2}{n(n-1)}\cdot d_{2}(n-d_{1}%
-d_{2}-1)d_{1}.
\]
From here we obtain
\[
\bar{l}(G^{\prime})-\pi(G^{\prime})\geq\bar{l}(G)-\pi(G)+\frac{d_{1}d_{2}%
}{n-1}\left(  \frac{2(n-d_{1}-d_{2}-1)}{n}-1\right)  .
\]
By Lemma \ref{lemma0Reference} we have $n-d_{1}-d_{2}-1\geq\frac{n}{2},$
therefore $\bar{l}(G^{\prime})-\pi(G^{\prime})\geq\bar{l}(G)-\pi(G)$.

CASE II: $u=v.$ Obviously, $v$ is the only branching vertex in $G.$ Therefore
$G$ consists of paths with common end point $v.$ Let $P_{1}$ and $P_{2}$ be
two shortest such path. If $V\backslash(P_{1}\cup P_{2}\cup\left\{  v\right\}
)$ contains at least $\frac{n}{2}$ vertices, then we make the same argument as
in case I. Otherwise $G$ is a tree consisting of four paths of equal length
with a common end point and the claim follows by Lemma \ref{lemma1Cross}.

Applying the transformations from cases I and II repeatedly, one obtains the claimed.
\end{proof}

\begin{figure}[h]
\begin{center}
\raisebox{-1\height}{\includegraphics[scale=0.3]{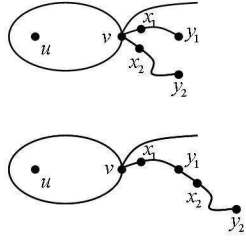}}
\end{center}
\caption{Tree transformation in the proof of Lemma \ref{lemma1General}.}%
\label{Figure1}%
\end{figure}

\begin{lemma}
\label{lemma1Balancing}Among trees with three leafs, the difference $\bar
{l}-\pi$ is maximum for a tree $G$ on $n$ vertices consisting of three paths
of almost equal length with a common end point.
\end{lemma}

\begin{proof}
Let $G$ be a tree with three leafs. That implies $G$ consists of three paths
with a common end vertex. Let $u$ be centroidal vertex of $G$, let $v$ be the
branching vertex furthest from $u.$ If $u\not =v,$ then by the same argument
as in case I of the proof of Lemma \ref{lemma1General} we obtain that the
difference $\bar{l}-\pi$ is greater or equal for path $P_{n}$ than for $G.$
Now the claimed follows from Lemma \ref{lemma1Path}. Else if $u=v,$ then all
three paths graph $G$ consists of have less than $\frac{n}{2}$ vertices. Let
$v_{1}$ be the leaf furthest from $u$ and $v_{2}$ leaf closest to $u.$ Let
$G^{\prime}$ be a tree obtained from $G$ by deleting edge incident to $v_{1}$
and adding edge $v_{1}v_{2}.$ We want to prove that the difference $\bar
{l}-\pi$ is greater or equal for $G^{\prime}$ than for $G.$ Let $d_{1}%
=d(u,v_{1})$ and $d_{2}=d(u,v_{2}).$ We have
\[
\pi^{\prime}(u)=\pi(u)-\frac{d_{1}-d_{2}-1}{n-1}.
\]
Also
\[
\bar{l}(G^{\prime})=\bar{l}(G^{\prime})-\frac{2}{n(n-1)}(n-d_{1}%
-d_{2}-1)\left(  d_{1}-d_{2}-1\right)  .
\]
From here we obtain
\[
\bar{l}(G^{\prime})-\pi(G^{\prime})\geq\bar{l}(G)-\pi(G)+\frac{d_{1}-d_{2}%
-1}{n-1}\left(  1-\frac{2}{n}(n-d_{1}-d_{2}-1)\right)  .
\]
Since all three paths $G$ consists of have less then $\frac{n}{2}$ vertices,
we can conclude that $n-d_{1}-d_{2}-1\leq\frac{n}{2}$ from which follows
$\bar{l}(G^{\prime})-\pi(G^{\prime})\geq\bar{l}(G)-\pi(G).$ By repeating this
tree transformation we obtain the claim.
\end{proof}

We can summarize the results of Lemmas \ref{lemma1Path}, \ref{lemma1Cross},
\ref{lemma1General} and \ref{lemma1Balancing} into following theorem.

\begin{theorem}
\label{tm1Con1Trees}Among all trees on $n\geq3$ vertices with average distance
$\bar{l}$ and proximity $\pi$, the difference $\bar{l}-\pi$ is maximum for a
tree $G$ composed of three paths of almost equal lengths with a common end point.
\end{theorem}

Therefore, we have proved Conjecture \ref{Con1} for trees. If for every graph
we find a tree for which difference $\bar{l}-\pi$ is greater or equal, the
Conjecture \ref{Con1} for general graphs will follow from Theorem
\ref{tm1Con1Trees}.

\begin{theorem}
\label{tm1Con1Gen}Among all connected graphs $G$ on $n\geq3$ vertices with
average distance $\bar{l}$ and proximity $\pi$, the difference $\bar{l}-\pi$
is maximum for a graph $G$ composed of three paths of almost equal lengths
with a common end point.
\end{theorem}

\begin{proof}
Let $G$ be a connected graph on $n\geq3$ vertices and let $u\in V(G)$ be a
vertex in $G$ such that $\pi(u)=\pi(G).$ Let $G^{\prime}$ be a breadth-first
search tree of $G$ rooted at $u$. Obviously, $\pi(G)=\pi(u)=\pi^{\prime
}(u)\geq\pi(G^{\prime}).$ As for $\bar{l},$ by deleting edges from $G$
distances between vertices can only increase, therefore $\bar{l}(G)\leq\bar
{l}(G^{\prime}).$ Now we have $\bar{l}(G)-\pi(G)\leq\bar{l}(G^{\prime}%
)-\pi(G^{\prime})$ and the claim follows from Theorem \ref{tm1Con1Trees}.
\end{proof}

\section{Average eccentricity and remoteness}

Now, let us find maximal trees for $ecc-\rho,$ proving thus that Conjecture
\ref{Con2} holds for trees.

\begin{lemma}
\label{lemma2Repic}Let $G$ be a tree on $n$ vertices with diameter $D$ and let
$P=v_{0}v_{1}\ldots v_{D}$ be a diametric path in $G.$ If there is $j\leq D/2$
such that the degree of $v_{k}$ is at most $2$ for $k\geq j+1,$ then the
difference $ecc-\rho$ is greater or equal for path $P_{n}$ than for $G.$
\end{lemma}

\begin{proof}
Let $w$ be a leaf in $G$ distinct from $v_{0}$ and $v_{D}.$ Let $G^{\prime}$
be a tree obtained from $G$ by deleting edge incident to $w$ and adding edge
$v_{D}w.$ Note that $G^{\prime}$ has diameter $D+1.$ We want to prove that
difference $ecc-\rho$ did not decrease by this transformation. First note that
eccentricity increased by $1$ for at least $n-\frac{D+1}{2}$ vertices.
Therefore, $ecc^{\prime}\geq ecc+\frac{2n-D-1}{2n}.$ As for remoteness, first
note that $\rho(G)=\pi(v_{D})$ and $\rho(G^{\prime})=\pi^{\prime}(w).$ Now,
let $d_{w}$ be the distance between vertices $w$ and $v_{D}$ in $G,$ i.e.
$d_{w}=d(w,v_{D}).$ Obviously $d_{w}\geq\frac{D+2}{2}.$ Now, we have
\[
\pi^{\prime}(w)=\pi(v_{D})+\frac{n-d_{w}-1}{n-1}\leq\pi(v_{D})+\frac
{2n-D-3}{2(n-1)}.
\]
Therefore,
\[
ecc^{\prime}-\rho^{\prime}\geq ecc-\rho+\frac{2n-D-1}{2n}-\frac{2n-D-3}%
{2(n-1)}\geq ecc-\rho.
\]
Repeating this transformation, we obtain the claim.
\end{proof}

\begin{theorem}
\label{tm2Con2trees}Among trees on $n\geq3$ vertices, the difference
$ecc-\rho$ is maximum for path $P_{n}.$
\end{theorem}

\begin{proof}
Let $G$ be a tree on $n$ vertices and diameter $D.$ Let $P=v_{0}v_{1}\ldots
v_{D}$ be a diametric path in $G.$ Let $G_{i}$ be a tree that is connected
component of $G\backslash P$ rooted in $v_{i}$ and let $V_{i}$ be the vertex
set of $G_{i}.$ If there is $j\leq D/2$ such that the degree of $v_{k}$ is at
most $2$ for $k\geq j+1,$ then the claim follows from Lemma \ref{lemma2Repic}.
Else, let $v_{j}$ and $v_{k}$ be vertices on $P$ of degree at least $3$ such
that $j\leq\frac{D}{2}<k$ and $k-j$ is minimum possible. Let $w_{j}$ be a
vertex outside of $P$ adjacent to $v_{j}$ and let $w_{k}$ be a vertex outside
of $P$ adjacent to $v_{k}.$ Let $G^{\prime}$ be a tree obtained from $G$ so that:

1) for every vertex $w$ adjacent to $v_{j},$ except $w=w_{j}$ and $w=v_{j+1},
$ edge $wv_{j}$ is deleted and edge $ww_{j}$ aded,

2) for every vertex $w$ adjacent to $v_{k},$ except $w=w_{k}$ and $w=v_{k-1},
$ edge $wv_{k}$ is deleted and edge $ww_{k}$ aded.

\noindent This transformation is illustrated with Figure \ref{Figure2}. Note
that diameter of $G^{\prime}$ equals $D+2.$ We want to prove that
$ecc^{\prime}-\rho^{\prime}\geq ecc-\rho.$ For that purpose, let us denote
\begin{align*}
V_{j}^{\prime}  &  =\left\{  v\in V_{j}:d(v,w_{j})<d(v,v_{j})\right\}  ,\\
V_{k}^{\prime}  &  =\left\{  v\in V_{k}:d(v,w_{k})<d(v,v_{k})\right\}  .
\end{align*}
Now, let us introduce following partition of set of vertices $V$%
\begin{align*}
X_{1}  &  =V_{0}\cup\ldots\cup V_{j-1}\cup(V_{j}\backslash(V_{j}^{\prime}%
\cup\left\{  v_{j}\right\}  )),\\
X_{2}  &  =V_{j}^{\prime},\\
X_{3}  &  =\left\{  v_{j}\right\}  \cup V_{j+1}\cup\ldots\cup V_{k-1}%
\cup\left\{  v_{k}\right\}  ,\\
X_{4}  &  =V_{k}^{\prime},\\
X_{5}  &  =\left(  V_{k}\backslash(V_{k}^{\prime}\cup\left\{  v_{k}\right\}
)\right)  \cup V_{k+1}\cup\ldots\cup V_{D}.
\end{align*}
Let $x_{i}=\left\vert X_{i}\right\vert .$ Now, let us compare $e^{\prime}(v) $
and $e(v)$ for every vertex $v\in V.$ Note that for $v\in X_{2}\cup X_{3}\cup
X_{4}$ holds $e^{\prime}(v)=e(v)+1,$ while for $v\in X_{1}\cup X_{5}$ holds
$e^{\prime}(v)=e(v)+2.$ Therefore,
\[
ecc^{\prime}=ecc+\frac{2x_{1}+x_{2}+x_{3}+x_{4}+2x_{5}}{n}=ecc+\Delta_{1}.
\]
Now, we want to compare $\pi^{\prime}(v)$ and $\pi(v)$ for every $v\in V.$ We
distinguish several cases depending whether $v\in X_{1},$ $v\in X_{2},$ $v\in
X_{3},$ $v\in X_{4}$ or $v\in X_{5}$. It is sufficient to consider cases $v\in
X_{1},$ $v\in X_{2}$ and $v\in X_{3},$ since $v\in X_{4}$ is analogous to
$v\in X_{2}$ and $v\in X_{5}$ is analogous to $v\in X_{1}.$

If $v\in X_{1},$ then the difference $d^{\prime}(v,u)-d(v,u)$ equals $0$ for
$u\in X_{1},$ equals $-1$ for $u\in X_{2},$ equals $1$ for $u\in X_{3}\cup
X_{4}$ and equals $2$ for $u\in X_{5}.$ Therefore,
\[
\pi^{\prime}(v)=\pi(v)+\frac{-x_{2}+x_{3}+x_{4}+2x_{5}}{n-1}=\pi(v)+\Delta
_{2}.
\]
If $v\in X_{2},$ then the difference $d^{\prime}(v,u)-d(v,u)$ equals $-1$ for
$u\in X_{1},$ equals $0$ for $u\in X_{2}\cup X_{3}\cup X_{4}$ and equals $1$
for $u\in X_{5}.$ Therefore,
\[
\pi^{\prime}(v)=\pi(v)+\frac{-x_{1}+x_{5}}{n-1}=\pi(v)+\Delta_{3}.
\]
If $v\in X_{3},$ then the difference $d^{\prime}(v,u)-d(v,u)$ equals $1$ for
$u\in X_{1}\cup X_{5}$ and equals $0$ for $u\in X_{2}\cup X_{3}\cup X_{4}.$
Therefore,%
\[
\pi^{\prime}(v)=\pi(v)+\frac{x_{1}+x_{5}}{n-1}=\pi(v)+\Delta_{4}.
\]
It is easily verified that $\Delta_{1}-\Delta_{2}\geq0,$ $\Delta_{1}%
-\Delta_{3}\geq0$ and $\Delta_{1}-\Delta_{4}\geq0$, so for every $v\in V$ we
obtain $ecc^{\prime}-\pi^{\prime}(v)\geq ecc-\pi(v).$

Now, let $u\in V$ be a vertex for which $\pi^{\prime}(u)=\rho(G^{\prime}).$ We
have
\begin{align*}
ecc(G^{\prime})-\rho(G^{\prime})  &  =ecc(G^{\prime})-\pi^{\prime}(u)\geq
ecc(G)-\pi(u)\geq\\
&  \geq ecc(G)-\max\left\{  \pi(v):v\in V\right\}  =ecc(G)-\rho(G).
\end{align*}

\end{proof}

\begin{figure}[h]
\begin{center}
\raisebox{-1\height}{\includegraphics[scale=0.4]{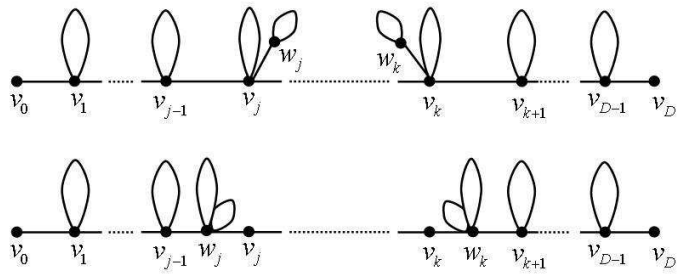}}
\end{center}
\caption{Tree transformation in the proof of Theorem \ref{tm2Con2trees}.}%
\label{Figure2}%
\end{figure}

Therefore, we have proved that $P_{n}$ is the tree which maximizes the
difference $ecc-\rho.$ Now, from
\[
ecc(P_{n})-\rho(P_{n})=\left\{
\begin{tabular}
[c]{ll}%
$\frac{n-2}{4}$ & for even $n,$\\
$\frac{n}{4}-\frac{2n+1}{4n}$ & for odd $n.$%
\end{tabular}
\right.
\]
easily follows that Conjecture \ref{Con2} holds for trees.

\section{Remoteness and radius}

First, we want to find minimal trees for $\rho-r.$ For that purpose, first
step is to reduce the problem to caterpillar trees.

\begin{lemma}
\label{lemma3Caterpillar}Let $G$ be a tree on $n$ vertices. There is a
caterpillar tree $G^{\prime}$ on $n$ vertices for which the difference
$\rho-r$ is less or equal than for $G$.
\end{lemma}

\begin{proof}
Let $P=v_{0}v_{1}\ldots v_{D}$ be a diametric path in $G.$ Let $G_{i}$ be a
tree that is connected component of $G\backslash P$ rooted in $v_{i}$ and let
$V_{i}$ be the vertex set of $G_{i}.$ Let $G^{\prime}$ be a caterpillar tree
obtained from $G$ in a following manner. In a tree $G_{i}$ let $v$ be the
non-leaf vertex furthest from $v_{i}$, let $w_{1},\ldots,w_{k}$ be all leafs
adjacent to $v,$ and let $u$ be the only remaining vertex adjacent to $v.$
Now, for every $j=1,\ldots,k$ edge $w_{j}v$ is deleted and edge $w_{j}u$ is
added. This transformation is illustrated with Figure \ref{Figure3}.The
procedure is done repeatedly in every $G_{i}$ ($2\leq i\leq D-2$) until the
caterpillar tree $G^{\prime}$ is obtained. Note that $G^{\prime}$ has the same
diameter (and therefore radius) as $G.$ What remains to be proved is that
remoteness in $G^{\prime}$ is less or equal than in $G.$ It is sufficient to
prove that the described transformation does not increase remoteness.
Obviously, $\pi^{\prime}(u)\leq\pi(u)$ for every $u\in V\backslash\left\{
v\right\}  .$ Number $\pi^{\prime}(v)$ can be greater than $\pi(v),$ but note
that $\pi^{\prime}(v)=\pi^{\prime}(w_{i})\leq\pi(w_{i})\leq\rho.$ Therefore
$\rho^{\prime}\leq\rho.$
\end{proof}

\begin{figure}[h]
\begin{center}
\raisebox{-1\height}{\includegraphics[scale=0.4]{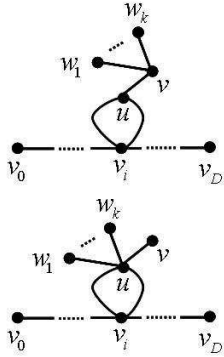}}
\end{center}
\caption{Tree transformation in the proof of Lemma \ref{lemma3Caterpillar}.}%
\label{Figure3}%
\end{figure}

Now that we reduced the problem to caterpillar trees, let us prove some
auxiliary results for such trees. First note that because of Lemma
\ref{lemma0Reference}, a leaf in a tree can not be centroidal vertex.
Therefore, in a caterpillar tree a centroidal vertex must be on diametric path
$P.$

\begin{lemma}
\label{lemma3Aux1}Let $G\not =P_{n}$ be a caterpillar tree on $n$ vertices
with diameter $D$, remoteness $\rho$ and only one centroidal vertex. Let
$P=v_{0}v_{1}\ldots v_{D}$ be a diametric path in $G$ such that $v_{j}\in P$
is the only centroidal vertex in $G$ and every of the vertices $v_{j+1}%
,\ldots,v_{D}$ is of degree at most $2.$ Then there is a caterpillar tree
$G^{\prime}$ on $n$ vertices of diameter $D+1$ and remoteness at most
$\rho+\frac{1}{2}.$
\end{lemma}

\begin{proof}
If $v_{j}$ is of degree $2,$ then by Lemma \ref{lemma0Reference} follows that
$j\leq\frac{D}{2}$, so the $\rho=\pi(v_{D}).$ Let $w$ be any leaf in $G$
distinct from $v_{0}$ and $v_{D}.$ Let $G^{\prime}$ be a graph obtained from
$G$ by first deleting edge incident to $w,$ then deleting edge $v_{j-1}v_{j}$
and adding path $v_{j-1}wv_{j}$ instead. This transformation is illustrated in
Figure \ref{Figure4}. Note that diameter of $G^{\prime}$ is $D+1,$ while
remoteness is still obtained for $v_{D}.$ Note that distances from $v_{D}$
have increased by $1$ for at most $\frac{n}{2}-1$ vertices. Therefore,
$\pi^{\prime}(v_{D})\leq\pi(v_{D})+\frac{n-2}{2(n-1)}$ from which follows
$\rho^{\prime}\leq\rho+\frac{1}{2}$ and the claim is proved in this case.

If degree of $v_{j}$ is greater than $2,$ let $w$ be a leaf on $v_{j}$, let
$V_{L}=V_{1}\cup\ldots\cup V_{j-1}$ and $V_{R}=V_{j+1}\cup\ldots\cup V_{D}.$
Since $v_{j}$ is centroidal vertex, from Lemma \ref{lemma0Reference} follows
that $V_{L}$ and $V_{R}$ have at most $\frac{n}{2}$ vertices. If any of them
had exactly $\frac{n}{2}$ vertices, then $G$ would have two centroidal
vertices by Lemma \ref{lemma0Reference}, which would be contradiction with
$v_{j}$ being only centroidal vertex. Therefore, we conclude $\left\vert
V_{L}\right\vert \leq\frac{n-1}{2}$ and $\left\vert V_{R}\right\vert \leq
\frac{n-1}{2}.$ Now it is possible to divide set of vertices $V_{j}%
\backslash\left\{  v_{j}\right\}  $ into two subsets $V_{j}^{\prime}$ and
$V_{j}^{\prime\prime}$ such that $\left\vert V_{L}\cup V_{j}^{\prime
}\right\vert \leq\frac{n-1}{2}$ and $\left\vert V_{R}\cup V_{j}^{\prime\prime
}\right\vert \leq\frac{n-1}{2}.$ Let $G^{\prime}$ be a graph obtained from $G$
by first deleting edge incident to $w,$ then deleting edge $v_{j}v_{j+1}$ and
adding path $v_{j}wv_{j+1}$ instead, and finally for every vertex $v\in
V_{j}^{\prime\prime}$ edge $vv_{j}$ is deleted and edge $vw$ added. This
transformation is illustrated in Figure \ref{Figure4}. Note that diameter of
$G^{\prime}$ is $D+1.$ Now, if $v\in V_{L}\cup V_{j}^{\prime}\cup\left\{
v_{j},w\right\}  $ the distance $d(v,u)$ is increased by $1$ only if $u\in
V_{R}\cup V_{j}^{\prime\prime},$ therefore $\pi^{\prime}(v)\leq\pi(v)+\frac
{1}{2}.$ If $v\in V_{R}\cup V_{j}^{\prime\prime}$ the distance $d(v,u)$ is
increased by $1$ only if $u\in V_{L}\cup V_{j}^{\prime},$ therefore
$\pi^{\prime}(v)\leq\pi(v)+\frac{1}{2}.$ We conclude $\rho^{\prime}\leq
\rho+\frac{1}{2},$ and the claim is proved in this case too.
\end{proof}

\begin{figure}[h]
\begin{center}
$%
\begin{array}
[c]{cccc}%
\text{a)} &
\text{\raisebox{-1\height}{\includegraphics[scale=0.3]{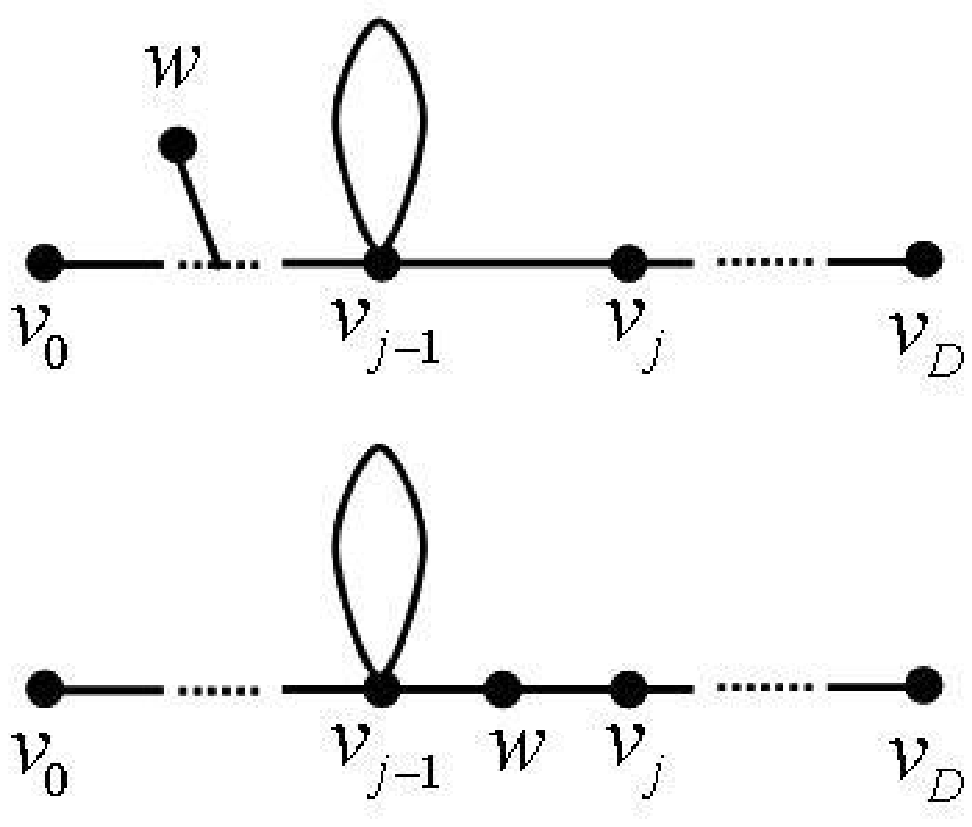}}} &
\text{ \ \ b)} &
\text{\raisebox{-1\height}{\includegraphics[scale=0.3]{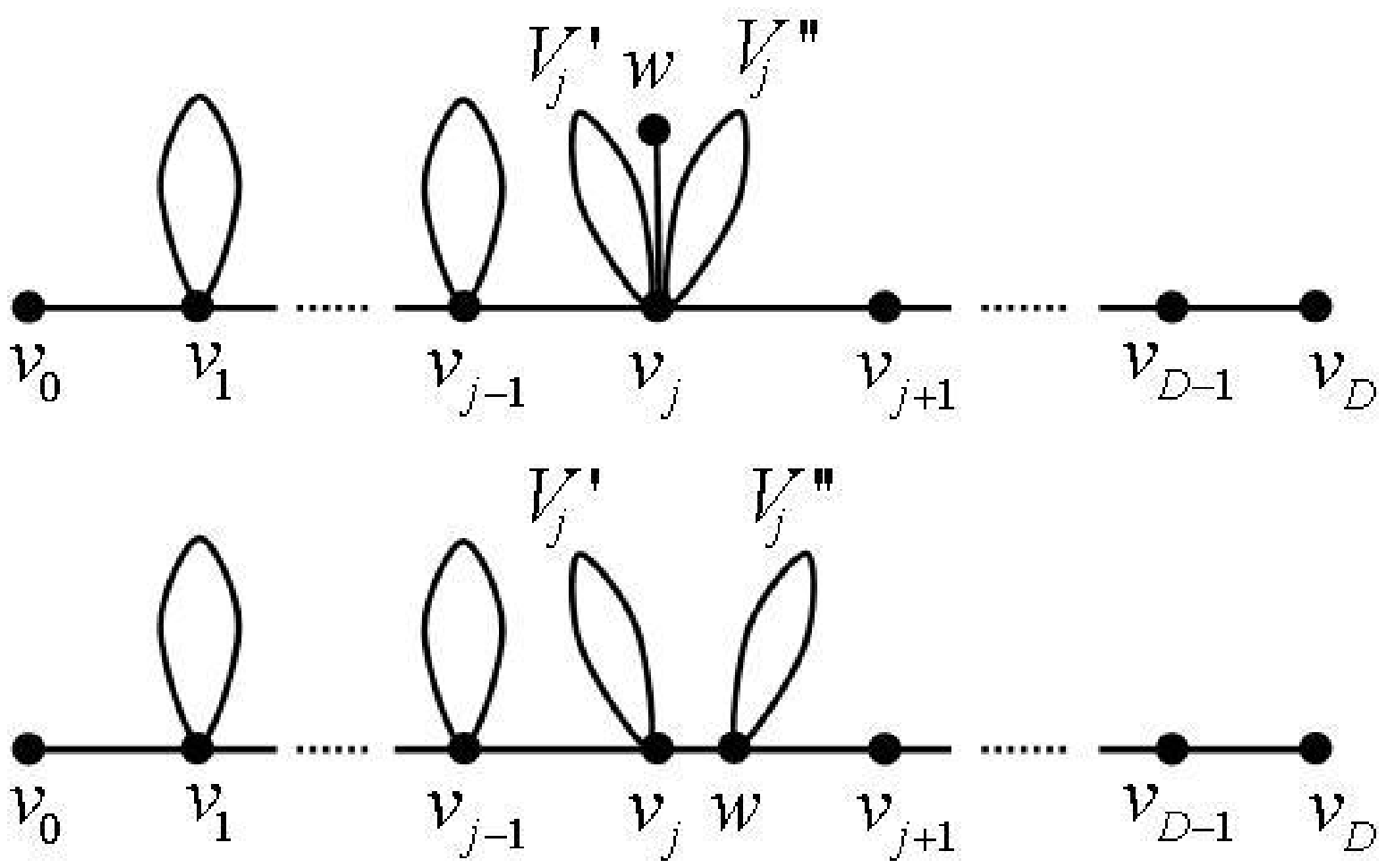}}}%
\end{array}
$
\end{center}
\caption{Tree transformations in the proof of Lemma \ref{lemma3Aux1}: a)
$v_{j}$ is of degree $2$, b) $v_{j}$ is of degree at least $3$.}%
\label{Figure4}%
\end{figure}

\begin{lemma}
\label{lemma3Aux2}Let $G\not =P_{n}$ be a caterpillar tree on $n$ vertices
with diameter $D$, remoteness $\rho$ and exactly two centroidal vertices. Let
$P=v_{0}v_{1}\ldots v_{D}$ be a diametric path in $G$ such that $v_{j}%
,v_{j+1}\in P$ are centroidal vertices and every of the vertices
$v_{j+1},\ldots,v_{D}$ is of degree at most $2.$ Then there is a caterpillar
tree $G^{\prime}$ on $n$ vertices of diameter $D+1$ and remoteness at most
$\rho+\frac{1}{2}.$
\end{lemma}

\begin{proof}
Since $v_{j+1}$ is centroidal vertex, from Lemma \ref{lemma0Reference} follows
that $j\leq\frac{D}{2}$, so $\rho=\pi(v_{D}).$ Let $w$ be any leaf in $G$
distinct from $v_{0}$ and $v_{D}.$ Let $G^{\prime}$ be a graph obtained from
$G$ by first deleting edge incident to $w,$ then deleting edge $v_{j}v_{j+1}$
and adding path $v_{j}wv_{j+1}$ instead. The diameter of $G^{\prime}$ is $D+1$
and remoteness is still obtained for $v_{D}.$ Note that distances from $v_{D}$
increased by $1$ for at most $\frac{n}{2}-1$ vertices, so $\pi^{\prime}%
(v_{D})\leq\pi(v_{D})+\frac{n-2}{2(n-1)}$. Therefore, $\rho^{\prime}\leq
\rho+\frac{1}{2}.$
\end{proof}

\begin{lemma}
\label{lemma3Aux3}Let $G\not =P_{n}$ be a caterpillar tree on $n$ vertices
with diameter $D$, remoteness $\rho$ and exactly two centroidal vertices of
different degrees. Let $P=v_{0}v_{1}\ldots v_{D}$ be a diametric path in $G$
such that $v_{j},v_{j+1}\in P$ are centroidal vertices and every of the
vertices $v_{0},\ldots,v_{j-1},v_{j+2},\ldots,v_{D}$ is of degree at most $2.$
Then there is a caterpillar tree $G^{\prime}$ on $n$ vertices of diameter
$D+1$ and remoteness at most $\rho+\frac{1}{2}.$
\end{lemma}

\begin{proof}
Let $d_{1}=d(v_{0},v_{j})$ and $d_{2}=d(v_{j+1},v_{D}).$ Without loss of
generality we may assume that $d_{1}\leq d_{2}.$ Since the degrees of $v_{j}$
and $v_{j+1}$ differ, from Lemma \ref{lemma0Reference} we conclude
$d_{1}\not =d_{2}.$ Therefore, $d_{1}<d_{2}.$ From this follows $j+1\leq
\frac{D}{2},$ so $\rho=\pi(v_{D}).$ Let $G^{\prime}$ be a graph obtained from
$G $ so that for every leaf $w$ incident to $v_{j}$ (distinct from $v_{0}$) we
delete edge $wv_{j}$ and add edge $wv_{j+1}.$ The diameter of $G^{\prime} $ is
still $D,$ while the remoteness $\rho^{\prime}$ is less or equal than $\rho.$
Now the claim follows from Lemma \ref{lemma3Aux2} applied on $G^{\prime}.$
\end{proof}

\begin{lemma}
\label{lemma3Aux4}Let $G\not =P_{n}$ be a caterpillar tree on $n$ vertices
with diameter $D$, remoteness $\rho$ and exactly two centroidal vertices of
equal degrees. Let $P=v_{0}v_{1}\ldots v_{D}$ be a diametric path in $G$ such
that $v_{j},v_{j+1}\in P$ are centroidal vertices and every of the vertices
$v_{0},\ldots,v_{j-1},v_{j+2},\ldots,v_{D}$ is of degree at most $2.$ Then the
difference $\rho-r$ is less or equal for path $P_{n}$ than for $G$.
\end{lemma}

\begin{proof}
Let $d_{1}=d(v_{0},v_{j})$ and $d_{2}=d(v_{j+1},v_{D}).$ Since $v_{j}$ and
$v_{j+1}$ have equal degrees, and every of the vertices $v_{0},\ldots
,v_{j-1},v_{j+2},\ldots,v_{D}$ is of degree at most $2,$ we conclude that
$d_{1}=d_{2}.$ Now, we will transform the tree twice which is illustrated with
Figure \ref{Figure5}. First, since $G$ is not a path, both $v_{j}$ and
$v_{j+1}$ must have a pendent leaf. Denote those leafs with $w_{1}$ and $w_{2}
$ respectively. Let $G^{\prime}$ be a graph obtained from $G$ by first
deleting edges incident to $w_{1}$ and $w_{2},$ then deleting edge
$v_{j}v_{j+1}$ and adding path $v_{j}w_{1}w_{2}v_{j+1}$ instead. Note that
$D^{\prime}=D+2.$ Therefore, $r^{\prime}=r+1.$ Also, note that remoteness in
both $G$ and $G^{\prime}$ is obtained for $v_{0}$ and $v_{D}$. Since distances
from $v_{0}$ have increased by $2$ for at most $\frac{n}{2}-1$ vertices, we
conclude $\pi^{\prime}(v_{0})\leq\pi(v_{0})+\frac{2(n-2)}{2(n-1)}$ from which
follows $\rho^{\prime}\leq\rho+1.$ Thus we obtain $\rho^{\prime}-r^{\prime
}\leq\rho-r.$ If $G^{\prime}$ is a path, then the claim is proved. Else, we
transform $G^{\prime}$ so that for every leaf $w$ in $G^{\prime}$ incident to
$v_{j}$ edge $wv_{j}$ is deleted and edge $ww_{1}$ is added. Also, for every
leaf $w$ in $G^{\prime}$ incident to $v_{j+1}$ edge $wv_{j+1}$ is deleted and
edge $ww_{2}$ is added. Note that this transformation changes neither radius
neither remoteness. Thus we obtain the tree on which we can repeat the whole
procedure. After repeating procedure finite number of times we obtain path
$P_{n}$ and the claim is proved.
\end{proof}

\begin{figure}[h]
\begin{center}
\raisebox{-1\height}{\includegraphics[scale=0.3]{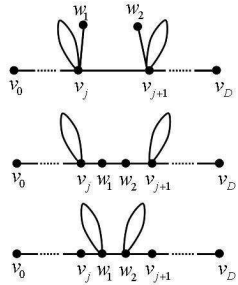}}
\end{center}
\caption{Tree transformations in the proof of Lemma \ref{lemma3Aux4}. }%
\label{Figure5}%
\end{figure}

Now that we have established auxiliary results for caterpillar trees, we can
find minimal trees for $\rho-r$ among caterpillar trees.

\begin{lemma}
\label{lemma3General}Let $G$ be a caterpillar tree on $n$ vertices. If $n$ is
odd, then the difference $\rho-r$ is less or equal for path $P_{n}$ then for
$G.$ If $n$ is even, then the difference $\rho-r$ is less or equal for path
$P_{n-1}$ with a leaf appended to a central vertex than for $G.$
\end{lemma}

\begin{proof}
Let $D$ be the diameter in $G$ and let $P=v_{0}v_{1}\ldots v_{D}$ be the
diametric path in $G.$ Suppose $D\leq n-3.$ That means $G$ has at least two
leafs outside $P.$ Let $v_{j}\in P$ be a centroidal vertex in $G.$ If there
are two vertices $v_{k}$ and $v_{l}$ on $P$ ($k<j<l$) with a pendent leaf on
them (distinct from $v_{0}$ and $v_{D}$), then the caterpillar tree
$G^{\prime}$ obtained from $G$ by deleting a leaf from $v_{j}$ and $v_{k}$ and
adding a leaf on $v_{j+1}$ and $v_{k-1}$ has the same radius and the
remoteness which is less or equal then in $G$. By repeating this procedure, we
obtain a caterpillar tree $G^{\prime}$ of the same diameter as $G$ with
diametric path $P=v_{0}v_{1}\ldots v_{D}$ such that:

\begin{enumerate}
\item $G^{\prime}$ has exactly one centroidal vertex $v_{j}\in P$ and every of
the vertices $v_{j+1},\ldots,v_{D}$ is of degree at most $2,$

\item $G^{\prime}$ has two centroidal vertices $v_{j},v_{j+1}\in P$ and every
of the vertices $v_{j+1},\ldots,v_{D}$ is of degree at most $2,$

\item $G^{\prime}$ has two centroidal vertices $v_{j},v_{j+1}\in P$ and every
of the vertices $v_{0},\ldots,v_{j-1},v_{j+2},\ldots,v_{D}$ is of degree at
most $2.$
\end{enumerate}

Therefore, on the obtained graph $G^{\prime}$ one of the Lemmas
\ref{lemma3Aux1}, \ref{lemma3Aux2}, \ref{lemma3Aux3} or \ref{lemma3Aux4} can
be applied. If Lemma \ref{lemma3Aux4} is applied, the claim is proved. Else if
Lemma \ref{lemma3Aux1}, \ref{lemma3Aux2} or \ref{lemma3Aux3} is applied, we
obtain graph $G^{\prime}$ of diameter $D+1$ and remoteness $\rho+\frac{1}{2}.$
Since for $D+1$ holds $D+1\leq n-2,$ we can apply the whole procedure with
$G=G^{\prime}$ (as the second step) and thus obtain a caterpillar tree
$G^{\prime}$ of diameter $D+2$ and remoteness $\rho^{\prime}\leq\rho+1.$ Since
for thus obtained $G^{\prime}$ holds $D^{\prime}=D+2,$ we conclude $r^{\prime
}=r+1.$ Therefore, $\rho^{\prime}-r^{\prime}\leq\rho-r.$

Repeating this double step, we obtain a caterpillar tree $G^{\prime}$ of
diameter $D^{\prime}=n-2$ or $D^{\prime}=n-1$ for which the difference
$\rho-r$ is less or equal than for $G.$ Now we distinguish several cases with
respect to $D^{\prime}$ and parity of $n.$ Suppose first $D^{\prime}=n-1.$
Then $G^{\prime}=P_{n}.$ If $n$ is odd then the claim is proved. If $n$ is
even it is easily verified that the difference $\rho-r$ is less for path
$P_{n-1}$ with a leaf appended to a central vertex than for $G^{\prime}=P_{n}$
and the claim is proved in this case too. Suppose now that $D^{\prime}=n-2.$
That means $G^{\prime}$ is a path $P_{n-1}$ with a leaf appended to one vertex
of $P_{n-1}.$ If $n$ is odd, then deleting the only leaf in $G^{\prime}$ to
extend it to $P_{n}$ increases radius by $1$ and remoteness by less than $1,$
so the claim holds. If $n$ is even, then deleting the leaf in $G^{\prime}$
outside $P_{n-1}$ and appending it to central vertex of $P_{n-1}$ preserves
the radius and decreases the remoteness. Therefore, the claim holds in this
case too.
\end{proof}

We can summarize results of these lemmas in the following theorem which gives
minimal trees for $\rho-r.$

\begin{theorem}
\label{tm3Con3Trees}Let $G$ be a tree on $n$ vertices. If $n$ is odd, then the
difference $\rho-r$ is less or equal for path $P_{n}$ then for $G.$ If $n$ is
even, then the difference $\rho-r$ is less or equal for path $P_{n-1}$ with a
leaf appended to a central vertex than for $G.$
\end{theorem}

\begin{proof}
Follows from Lemmas \ref{lemma3Caterpillar} and \ref{lemma3General}.
\end{proof}

For a path $P_{n}$ on odd number of vertices $n$ holds $\rho-r=\frac{1}{2}$
which, together with Theorem \ref{tm3Con3Trees}, obviously implies that trees
on odd number of vertices satisfy Conjecture \ref{Con3}. Now, let us consider
graph $G$ on even number of vertices $n$ consisting of a path $P_{n-1}$ with a
leaf appended to a central vertex. For $G$ holds $\rho-r=\frac{n}{2(n-1)}$
which implies that trees on even number of vertices satisfy Conjecture
\ref{Con3} too.

\section{Conclusion}

We have established that maximal tree for $\bar{l}-\pi$ is a tree composed of
three paths of almost equal lengths with a common end point. Thus, we proved
that Conjecture \ref{Con1} posed in \cite{Aouchiche(2011)networks} for general
graph holds for trees. Using reduction of a graph to a corresponding subtree,
this result enabled us to prove Conjecture \ref{Con1} for general graphs too.
Also, we established that maximal tree for $ecc-\rho$ is path $P_{n}$ and that
minimal tree for $\rho-r$ is path $P_{n}$ in case of odd $n$ and path
$P_{n-1}$ with a leaf appended to a central vertex in case of even $n.$ Since
for these extremal trees Conjectures \ref{Con2} and \ref{Con3} posed in
\cite{Aouchiche(2011)networks} hold, it follows that those conjectures hold
for trees.

\section{Acknowledgements}

Partial support of the Ministry of Science, Education and Sport of the
Republic of Croatia (grant. no. 083-0831510-1511) and of project GReGAS is
gratefully acknowledged.

\end{document}